\documentclass[reqno]{amsart}
\usepackage{amsmath}
\usepackage{amsfonts}
\usepackage{amssymb,enumerate}
\usepackage{amsthm}
\usepackage[all]{xy}
\usepackage{rotating}
\usepackage{hyperref}
\usepackage{color}
\theoremstyle{plain}
\newtheorem{lem}{Lemma}[section]
\newtheorem{cor}[lem]{Corollary}
\newtheorem{prop}[lem]{Proposition}
\newtheorem{thm}[lem]{Theorem}

\newtheorem*{mthm*}{Theorem}

\theoremstyle{definition}

\newtheorem{ex}[lem]{Example}
\newtheorem{question}[lem]{Question}

\newtheorem{para}[lem]{}

\newtheorem*{convention*}{Convention}

\newcommand{\id}{\operatorname{id}}

\newcommand{\HH}{\operatorname{H}}

\newcommand{\End}{\operatorname{End}}

\newcommand{\Ker}{\operatorname{Ker}}

\newcommand{\xra}{\xrightarrow}

\newcommand{\e}{\mathbf{e}}

\renewcommand{\geq}{\geqslant}
\renewcommand{\leq}{\leqslant}
\renewcommand{\ker}{\Ker}

\newcommand{\Hom}{\operatorname{Hom}}

\def\e{\mathrm{e}}

\def\E{\mathcal{E}}

\def\C{\mathcal{C}}
\def\K{\mathcal{K}}
\def\Der{\mathrm{Der}}
\def\End{\mathrm{End}}

\newcommand{\gr}{\mathsf{Gr}}
\def\SDer{{}^*\Der}
\def\SHom{{}^*\!\Hom} 
\def\SEnd{{}^*\End} 
\def\SConn{{}^*\mathrm{Conn}}
\newcommand{\N}{\mathbb{N}}
\newcommand{\Q}{\mathbb{Q}}
\def\Conn{\mathrm{Conn}}

\numberwithin{equation}{lem}

\begin{document}

\bibliographystyle{amsplain}

\title[Connections and na\"{i}ve lifting of DG modules]{Connections and na\"{i}ve lifting of DG modules}

\author{Saeed Nasseh}
\address{Department of Mathematical Sciences\\
Georgia Southern University\\
Statesboro, GA 30460, U.S.A.}
\email{snasseh@georgiasouthern.edu}

\author{Maiko Ono}
\address{Department of Mathematics, Okayama University, 3-1-1 Tsushima-naka, Kita-ku, Okayama 700-8530, Japan}
\email{onomaiko.math@okayama-u.ac.jp}

\author{Yuji Yoshino}
\address{Graduate School of Environmental, Life, Natural Science and Technology, Okayama University, Okayama 700-8530, Japan}
\email{yoshino@math.okayama-u.ac.jp}

\thanks{Y. Yoshino was supported by JSPS Kakenhi Grants 19K03448 and 24K06690.}

%\dedicatory{}

\keywords{Atiyah map, connection, curvature, derivation, DG algebra, DG module, diagonal ideal, differential module, enveloping DG algebra, free DG module, integrable connection, Kodaira-Spencer homomorphism, module of K\"{a}hler differentials, na\"{i}ve lifting, semifree differential module, semifree resolution, topological DG module.}
\subjclass[2020]{13N15, 16E45, 16W25, 16W50, 16W80.}

\begin{abstract}
In this paper, we generalize the notion of connections, which was introduced by Alain Connes in noncommutative differential geometry, to the differential graded (DG) homological algebra setting. Then, along a DG algebra homomorphism $A \to B$, where $B$ is assumed to be projective as an underlying graded $A$-module, we give necessary and sufficient conditions for a semifree DG $B$-module to be na\"{i}vely liftable to $A$ in terms of connections.
\end{abstract}

\maketitle

\tableofcontents

\section{Introduction}\label{sec20200314a}

The notion of connections originates from Riemannian geometry and, as far as we know, its history goes back to the nineteenth century in the work of Christoffel~\cite{christoffel}. This notion was studied since then by several mathematicians and physicists including Levi-Civita and Ricci, Cartan, Ehresmann, and Koszul~\cite{levi, cartan1, cartan2, ehresmann2, ehresmann1, koszul}, in the order of appearance, to name a few. It was also generalized by Connes~\cite{AC} to the context of noncommutative differential geometry and studied further by Cuntz and Quillen in~\cite{CQ}; see also the works of Eriksen and Gustavsen~\cite{E, EG}.\vspace{2mm}

On the other hand, the notion of naïve liftability of differential graded (DG) modules was introduced and studied intensively by the authors in~\cite{NOY-jadid}-\cite{NOY} as a refined version of the classical notions of lifting and weak lifting~\cite{auslander:lawlom, nasseh:lql, nassehyoshino, OY, yoshino}. The main purpose of inventing naïve liftability is to study some fundamental problems in homological commutative algebra regarding the vanishing of Ext modules; see~\cite[Appendix A]{NOY-jadid} or~\cite[Section 7]{NOY1} for a comprehensive discussion on this topic.\vspace{2mm}

In this paper, we further generalize the notion of Connes' connections to the DG homological algebra setting (compare to~\cite{NOY3}); this is the subject of Section~\ref{sec20250421a}, which along with its preceding sections, lays out the groundwork for stating and proving our main result of this paper, namely, Theorem~\ref{naive}. In this theorem, over a DG $R$-algebra homomorphism $A \to B$, where $R$ is a commutative ring and $B$ is projective as an underlying graded $A$-module, we give necessary and sufficient conditions for a semifree DG $B$-module to be na\"{i}vely liftable to $A$ in terms of connections along DG modules over the enveloping DG algebra $B^e$. Finally, in Section~\ref{sec20250421d}, we prove Theorem~\ref{fesOX} which includes other variants of some of the conditions discussed in Theorem~\ref{naive} by considering connections along DG $B$-modules (as opposed to DG $B^e$-modules); this theorem contains a condition that includes the Kodaira-Spencer homomorphism and might be of independent interest.

\section{Basic set up and completeness of topological Hom sets}\label{sec20250407a}

The purpose of this section is to specify some of the terminology from differential graded (DG, for short) homological algebra used throughout the paper along with recording some of the topological properties of the graded Hom set. General references on DG homological algebra include~\cite{avramov:ifr,avramov:dgha, felix:rht, GL}. Also, for the unspecified notations and details, the reader may consult either of the papers~\cite{NOY-jadid}, \cite{NOY3}, or~\cite{NOY1}.

\begin{para}
Throughout the paper, $R$ is a commutative ring and $(A,d^A)$, or simply $A$, is a \emph{DG $R$-algebra}. More precisely, $A = \bigoplus  _{n \geq 0} A _n$ is a non-negatively graded $R$-algebra which forms a chain complex with a differential $d^A$ (i.e., $d^A$ is a collection $\{d_n^A\colon A_n\to A_{n-1}\}_{n\in \mathbb{N}}$ of $R$-linear maps with $d^A_{n-1}\circ d^A_n=0)$ such that for all homogeneous elements $a \in A_{|a|}$ and $b \in A_{|b|}$ the following conditions hold:
\begin{enumerate}[\rm(i)]
\item
$ab = (-1)^{|a| |b|}ba$ and $a^2 =0$  if $|a|$ is odd;
\item
$d^A$ satisfies the \emph{Leibniz rule}, that is, $d^A(ab) = d^A(a) b + (-1)^{|a|}ad^A(b)$.
\end{enumerate}
Here, for the homogeneous element $a$, the notation $|a|$ stands for the degree of $a$.
\end{para}

\begin{para}\label{para20221103a}
In this paper, every DG module is assumed to be a \emph{right} DG module, unless otherwise is specified. More precisely, a \emph{DG $A$-module} $(M, \partial^M)$, or simply $M$, is a graded right $A$-module $M=\bigoplus_{n\in \mathbb{Z}}M_n$ which forms a chain complex with a differential $\partial^M$ that satisfies the \emph{Leibniz rule}, i.e., for all homogeneous elements $a\in A$ and $m\in M$ we have
$\partial^M(ma) = \partial^M(m)a + (-1)^{|m|} m d^A(a)$. The notion of left DG $A$-modules is defined similarly. Moreover, since by definition $A$ is a ``strongly commutative'' DG algebra, every right DG $A$-module $M$ has a structure of a left DG $A$-module by setting $am=(-1)^{|a||m|}ma$ for all homogeneous elements $a\in A$ and $m\in M$. We call $M=\bigoplus_{n\in \mathbb{Z}}M_{n}$ the \emph{underlying graded $A$-module}.

We denote by $\C(A)$ the category of DG $A$-modules and DG $A$-module homomorphisms. Isomorphisms in $\C(A)$ are denoted by $\cong$ and quasi-isomorphisms in $\C(A)$ are
identified by $\simeq$. Also, the homotopy category of $A$ is denoted by $\K(A)$.

%A DG $A$-module $M$ is \emph{bounded below} if $M_i=0$ for all $i\ll 0$.

Let $n$ be an integer. The \emph{$n$-th shift} of a DG $A$-module $M$, denoted $M(n)$, is defined to be the DG $A$-module $\left(M(n), \partial^{M(n)}\right)$, where for all integers $i$ we have $\left(M(n)\right)_i = M_{i+n}$ and $\partial_i^{M(n)}=(-1)^n\partial_{i+n}^M$. %We simply write $\shift M$ instead of $\shift^1 M$.
%Finally, recall that a free DG $A$-module $M$ is a graded free $A$-module with a \color{red} homogeneous?? \color{black} free basis $\{e_{\lambda}\mid \lambda\in \Lambda\}$ such that $\partial ^M (e_{\lambda})= 0$ for all $\lambda\in \Lambda$.
\end{para}

\begin{para}\label{para20250712a}
Note that $A(-n)$ is free as an underlying graded $A$-module. Thus, we have $A(-n)\cong eA$ in $\mathcal{C}(A)$, where $e$ is a free basis element of degree $n$ and $d^A(e)=0$.
%Sometimes $A(-n)$ is called a principal (?) free dg $A$-module with generator of degree $n$.
In general, a DG $A$-module $M$ is called \emph{free} if $M \cong \bigoplus _{\lambda\in \Lambda} e_{\lambda}A$, where $\{e_{\lambda}\mid \lambda\in \Lambda\}$ is a homogeneous free basis of $M$ as an underlying graded $A$-module and
$\partial^A(e_{\lambda})=0$ for all $\lambda\in \Lambda$. We refer to such a basis as a \emph{free basis} of the free DG $A$-module $M$.
\end{para}

\begin{para}\label{para20250407a}
We denote by $\gr(A)$ the abelian category of graded (right) $A$-modules and graded (i.e., degree-preserving) $A$-module homomorphisms.
For DG $A$-modules $M$ and $N$, we set 
$$
\SHom _A (M, N) = \bigoplus _{n\in \mathbb{Z}} \Hom _{\gr (A)} (M, N(n))
$$
where we refer to $\Hom _{\gr (A)} (M, N(n))$ as the set of graded $A$-module homomorphisms from $M$ to $N$ of degree $n$.
Note that $\SHom _A (M, N)$ is a DG $A$-module with the left $A$-module structure defined by $(af)(m)=af(m)=(-1)^{|a|(|f|+|m|)}f(m)a$, the right $A$-module structure given by $(fa)(m)=f(am)=(-1)^{|a||m|}f(m)a$, and with
%via
%the right $A$-action defined by the equality 
%$$
%(f \cdot a)(m) = f(ma)=f(m)a  
%$$
%and 
the differential defined by the formula
\begin{equation}\label{eq20250408b}
\partial ^{\Hom} (f) =\partial ^{N}  \circ f  - (-1)^{n} f \circ \partial ^M
\end{equation}
for all homogeneous elements $a\in A$, $m \in M$, and $f  \in \Hom _{\gr (A)} (M, N(n))$. %For simplicity, we denote $\Hom _{\gr (A)} (M, N(0))$ by $\Hom _{\gr (A)} (M, N)$.
\end{para}

\begin{para}\label{topology}
For DG $A$-modules $M$ and $N$ and a positive integer $n$, let
$$
U_n = \{ f \in \SHom _A (M, N)\mid f (M_i) =0  \ \text{for all} \ i \leq n \}.
$$ 
Then, the set $\{U_n\mid n \in \mathbb{N}\}$ forms a fundamental system of open neighborhoods for $0$ in $\SHom _A (M, N)$, which defines the linear topology on $\SHom _A (M, N)$. 
Note that the addition function $(f, f') \mapsto f+f'$ and the action function $f \mapsto f \cdot a$, for $a \in A$, on $\SHom _A (M, N)$ are continuous in this topology. 
It is also trivial that the differential $\partial ^{\Hom}$ is continuous. Hence, $\SHom _A (M, N)$ is a topological DG $A$-module. 
Also, since $\bigcap _{n \in \mathbb{N}} U_n ={0}$, any set consisting of a single element in $\SHom _A (M, N)$ is a closed subset in this topology. 
Therefore, the cycle  $\ker (\partial ^{\Hom})$  is a closed subset. 
\end{para}

In the next discussion, for the definition of semifree DG $A$-modules and more, see~\cite[A.2]{AINSW} or equivalently~\cite[2.12]{NOY1}.

\begin{para}\label{para20250407b}
Let $M$ be a semifree DG $A$-module with a semifree filtration $$0=F_0\subseteq F_1\subseteq F_2 \subseteq \cdots \subseteq M.$$ 
By definition, for each positive integer $n$, the quotient $F_n / F_{n-1}$ is a free DG $A$-module and we have $M = \bigcup _{n\in \N} F_n$. 
We additionally assume that the minimal degrees of free generators of each quotient $F_n / F_{n-1}$ increase as a function of $n$. 
In such a case, for positive integers $n$, we define the subsets
$$
V_n = \{ f \in \SHom _A (M, N)\mid f (F_n) =0\}
$$
of $\SHom _A (M, N)$. Then, $\{ V_n\mid n\in \N\}$ is also a fundamental system of open neighborhoods of $0$ in $\SHom _A (M, N)$ which defines the same topology as in~\ref{topology}. 
\end{para}

\begin{prop}\label{complete}
For DG $A$-modules $M$ and $N$, the topological DG $A$-module $\SHom _A (M, N)$ is complete and separated (that is, every Cauchy sequence converges to a single point).  
\end{prop}

\begin{proof}
Let $\{f_n\}_{n\in \N}$ be a Cauchy sequence in $\SHom _A (M, N)$, i.e., for every integer $n \geq 1$ there is an integer $m \geq 1$ such that
$f_i - f_j  \in U_n$ for all $i, j \geq m$. Thus, for a homogeneous element $x \in M$ with $|x| =n$ we have $f_m (x)=f_{m+1}(x)= \cdots$ for large enough $m$. 
Hence, we can define the mapping  $f :M \to N$ as  $f (x)$ being this common value for each $x$. 
Then, $\{ f_n \}_{n\in \N}$ uniquely converges to $f$, as desired.  
\end{proof}

We conclude this section by introducing the following notation.

\begin{para}\label{para20250408a}
Throughout the paper,  $A \to B$ is a DG $R$-algebra homomorphism, where $B$ is projective as an underlying graded $A$-module, and the enveloping DG $R$-algebra $B \otimes _A B$ is denoted by $B^e$. Also, $(J,d^J)$, or simply $J$, stands for the diagonal ideal of $B^e$, which by definition is the kernel of the map $\pi_B\colon B^e\to B$ defined by $\pi_B(b_1\otimes_Ab_2)=b_1b_2$. For more details on $B^e$ and $J$ see~\cite{NOY-jadid}--\cite{NOY1}. Moreover, we set $\Omega:=\Omega_A(B)=J/J^2$ that is called a \emph{differential module of $B$ over $A$}. If $A \to B$ is a homomorphism of commutative rings, then $\Omega$ is the module of K\"{a}hler differentials.
\end{para}

\section{Derivations}\label{sec20250408a}

In this section, we introduce the notion of derivations in the DG context and record some of their properties for using in the subsequent sections. Here, we use the notation that we established in the previous section. 

\begin{para}\label{para20250408b}
Let  $X$ be a DG $B^e$-module and $n$ be an integer. An {\it $A$-derivation} of degree $n$ is an element $D\in \Hom _{\gr(A)} (B, X(n))$ that satisfies the Leibniz rule, i.e., for all homogeneous elements $b_1, b_2 \in B$, the equality
\begin{equation}\label{eq20250421a}
D(b_1b_2) = D(b_1)b_2 + (-1) ^{|D||b_1|} b_1D(b_2) 
\end{equation}
holds with $|D|=n$. Notice that  $D(a)=0$ for an $A$-derivation $D$ and all $a\in A$. 
For the DG $B^e$-module $X$, the set of all $A$-derivations of degree $n$ is denoted by $\Der _A(B, X)_n$ and we set 
$$
\SDer _A(B, X) = \bigoplus _{n\in \mathbb{Z}} \Der _A(B, X)_n.   
$$
Note that $\SDer_A(B, X)$ has a structure of graded left $B$-module. In fact, for $D \in \Der _A (B, X) _n$ and $b \in B_m$, the $A$-derivation  $bD$ of degree $n +m$  is defined by 
$(bD) (c)= bD(c)$ for all $c \in B$. Moreover, $\SDer _A(B, X)$ has a left DG $B$-module structure with the differential defined by the formula
$$
\partial^{\Der} (D) = \partial ^X \circ D - (-1)^{|D|} D \circ d^B\in \Der _A(B, X)_{|D|-1}
$$
for all $D\in \Der _A(B, X)_{|D|}$. One can check that 
$$
\partial^{\Der} (bD) = d^B (b) D + (-1)^{|b|} b (\partial^{\Der} (D)) 
$$
for all homogeneous elements $b \in B$, that is, $\SDer _A(B, X)$ satisfies the Leibniz rule. 
\end{para}

%\begin{para}
%In the setting of~\ref{para20250408b}, the reader must be cautious. Let $D_1, D_2 \in \SDer_A(B, X)$ be homogeneous elements of different degrees. If $|D_1| \equiv |D_2| \pmod 2$, then $D_1+D_2$ satisfies the Leibniz rule. However, the element $D_1 +D_2$ %may not satisfies the Leibniz rule in general. In other words, $D \in \SDer _A(B, X)$ does not mean that $D$ necessarily satisfies the Leibniz rule unless it is a homogeneous element. 
%\end{para}

\begin{para}
In the setting of~\ref{para20250408b}, if  $X=B$, then we simply denote $\SDer_A(B, B)$ by $\SDer_A(B)$. This is a graded Lie algebra with the bracket product defined by  
\begin{equation}\label{eq20250729a}
[D_1, D_2] = D_1\circ D_2 -(-1)^{|D_1||D_2|} D_2\circ D_1
\end{equation}
for homogeneous elements $D_1,  D_2 \in \SDer_A(B)$.

We will use the notation of bracket product in different contexts throughout the paper with the same definition as in~\eqref{eq20250729a} and without further notice.
\end{para}

\begin{para}
Consider the setting of~\ref{para20250408b}, and let $D_1, D_2 \in \SDer_A(B, X)$ be homogeneous derivations with $|D_1|\neq |D_2|$. As a caution to the reader, note that if  $|D_1| \equiv |D_2| \pmod 2$, then $D_1+D_2$ satisfies the Leibniz rule~\eqref{eq20250421a}. However, the nonhomogeneous element $D_1 +D_2$ may not satisfy the Leibniz rule in general. In other words, $D \in \SDer _A(B, X)$ does not mean that $D$ satisfies~\eqref{eq20250421a} unless it is homogeneous.
\end{para}

\begin{para}
Considering the setting of~\ref{para20250408b}, similar to~\ref{topology}, a topology on $\SDer_A(B, X)$ can be defined. More  precisely, for each positive integer $n$, let  
$$
W_n = \{ D \in \SDer_A(B, X)\mid D(B_i) = 0 \ \text{for all}  \ i \leq n \}. 
$$
Then, the set $\{W_n\mid n \in \mathbb{N}\}$ forms a fundamental system of open neighborhoods for $0$ in $\SDer_A(B, X)$, which defines the linear topology on $\SDer_A(B, X)$.
\end{para}

The proof of the following result is similar to that of Proposition~\ref{complete}. 

\begin{prop}\label{prop20250408a}
Let $X$ be a DG $B^e$-module. Then, the topological DG $B$-module $\SDer_A(B, X)$ is complete and separated. 
\end{prop}

\begin{ex}\label{ex20250408a}
Let  $B = A [X_i\mid i\in \N]$  be a polynomial extension of $A$ obtained by adjunction of infinitely countably many variables such that $0<|X_1| \leq  |X_2| \leq  \cdots$ and that $\{ i \in \N\mid  |X_i|=n\}$ is a finite set for all positive integers $n$; see~\cite{NOY1} for more details. 
Then, 
$$
D_1 = \sum _{i=1}^\infty X_i \frac{\partial}{\partial X_i}  
$$ 
is an $A$-derivation in $\SDer _A(B)$ of degree $0$, that is,  $D_1 \in \Der _A(B)_0$. By Proposition~\ref{prop20250408a}, this infinite sum defines a unique derivation. 

On the other hand, notice that the element
$$
D_2 = \sum _{i=1}^\infty  \frac{\partial}{\partial X_i}  
$$
does not belong to $\SDer_A(B)$. 
\end{ex}

\begin{ex} 
The degree-preserving map $\E\colon B \to B$ defined by $\E (b) = |b| b$, for all homogeneous elements $b \in B$, is an element in $\Der _A(B)_0$ and is called the {\it Euler derivation}. Note that $B _0 \subseteq \Ker \E$, and if $R$ contains $\Q$, then equality holds.
\end{ex}

\begin{para}\label{para20250410a}
The $A$-linear map $\delta\colon B \to J$ defined by $\delta (b) = b \otimes_A 1 - 1 \otimes_A b$, for all $b \in B$, is an element of $\Der _A (B, J)_0$ and is called the \emph{universal derivation}; see~\cite{NOY3}. For the universal derivation $\delta \in \SDer_A(B, J)$ it always holds that  $\partial ^\Der (\delta)=0$. 
Hence, $\delta$ is a cycle in $\SDer_A(B, J)$ of degree $0$.

Given a DG $B^e$-module $X$ and an integer $n$, there is a natural isomorphism 
\begin{equation}\label{eq20250408a}
\varpi_n\colon \Hom _{\gr (B^e)}(J, X(n))\xra{\cong} \Der _A (B, X)_n 
\end{equation}
of abelian groups under which $f\in \Hom _{\gr (B^e)}(J, X(n))$ is mapped to 
\begin{equation}\label{eq20250410c}
\varpi_n(f)= f \circ \delta.
\end{equation} 
%$$
%\xymatrix{ B \ar[r] ^D \ar[dr] _{\delta} & X \\ &J \ar[u]_f } 
%$$
Note that $\SHom _{B^e}(J, X)= \bigoplus _{n\in \mathbb{Z}} \Hom _{\gr (B^e)}(J, X(n))$  is a left DG $B$-module defining the left $B$-action by $(bf )(j) = b f(j)$, for all $b \in B$, $f \in \SHom _{B^e}(J, X)$, and $j \in J$, and 
the differential by the formula~\eqref{eq20250408b}.
%$$
%\partial ^{\Hom} (f) = \partial ^X \circ f  - (-1)^{|f|} f \circ d^{J}
%$$
%for all $b \in B$, $f \in \Hom _{\gr(B^e)}(J, X(|f|))$, and $j \in J$. 

The isomorphisms $\varpi_n$ from~\eqref{eq20250408a} for all integers $n$ induce an isomorphism  
\begin{equation}\label{eq20250410b}
\varpi\colon \SHom _{B^e}(J, X) \xra{\cong}  \SDer_A (B, X)  
\end{equation}
of left DG $B$-modules. Furthermore, as an isomorphism between topological spaces, $\varpi$ is continuous.  
\end{para}

Recalling the notation from~\ref{para20250408a} and noting to the fact that $B^e/J\cong B$ as DG $R$-algebras, we summarize the above discussion in the following result. 

\begin{prop}\label{derivation}
Given a DG $B^e$-module $X$, we have an isomorphism 
$$
\SDer _A (B, X)\cong \SHom _{B^e} (J, X)
$$
of topological left DG $B^e$-modules. Taking homology we obtain an isomorphism 
$$
\HH\left(\SDer_A(B, X) \right) \cong \SHom _{\K(B^e)} (J, X) 
$$
of graded $R$-modules in which $\K(B^e)$ is the homotopy category of DG $B^e$-modules. 

Moreover, if $XJ=0$ (i.e., $X$ is a DG $B$-module), then there is an isomorphism 
\begin{equation}\label{eq20250410e}
\SDer _A (B, X) \cong \SHom _{B} (\Omega, X)  
\end{equation}
of topological DG $B$-modules. 
\end{prop}

\begin{para}
In Proposition~\ref{derivation}, if we assume that $X =J$, then 
$$
\SDer _A (B, J) \cong \SEnd _{B^e} (J)   
$$
in which $\SEnd _{B^e} (J)= \SHom _{B^e} (J, J)$ is a DG $R$-algebra that is non-commutative in general. This shows that $\SDer _A (B, J)$ is a DG module over the DG $R$-algebra $\SEnd_{B^e}(J)$. Furthermore, $\SDer _A (B, J)$ is a free DG module over $\SEnd _{B^e}(J)$ generated by the single element $\delta$. 
\end{para}

\begin{ex}\label{important}
Assume that either $B=A[X_i\mid i\in \N]$ is a polynomial extension of $A$, or $A$ is a divided power DG $R$-algebra and $B=A \langle X_i\mid i\in \N \rangle$ is a free extension of $A$; see~\cite{NOY1} for more details regarding the latter. In both cases we assume that the variables $X_i$ satisfy the same conditions as in Example~\ref{ex20250408a}, i.e., $0<|X_1| \leq  |X_2| \leq  \cdots$ and $\{ i \in \N\mid  |X_i|=n\}$ is a finite set for all integers $n\geq 1$. 

In~\cite[Theorem 6.10]{NOY1} we have shown that  $\Omega$ is a semifree DG $B$-module whose semifree basis is $\{u_i\mid i\in \N\}$, where each $u_i$ is the class of $\delta (X_i) \in J$ in $\Omega$. Then, there is a dual basis $\{\partial _i\mid i\in\N\}$ in $\SDer _A(B)=\SHom _{B} (\Omega, B)$, where each  $\partial _i\colon B \to B$ is an $A$-derivation such that
$$
\partial _i (X_j) =
\begin{cases}
1& i=j\\
0& i\neq j.     
\end{cases}
$$
Note that  $|\partial _i| = - |u_i| = -|X_i|$ are all negative and decreasing. Also, every element of $\SDer _A(B)$ is uniquely described as a power series of $\{\partial _i\mid i\in \N\}$. More precisely, for each element $D \in  \SDer _A(B)$ we have an equality
$$
D = \sum _{i=1}^{\infty} D(X_i) \partial _i  
$$
in which the power series converges because $\SDer _A(B)$ is complete. 
\end{ex}

\section{Connections}\label{sec20250421a}

The notion of connections was defined by Connes~\cite{AC} in the context of noncommutative differential geometry; see also~\cite{CQ}. A brief generalization of this notion in the DG setting is discussed in~\cite{NOY3}. Moreover, some of our discussions here on this topic can be considered as an upgraded version of the content in~\cite{E, EG} to the DG setting. Our main result in this section is Theorem~\ref{fundamental} in which we introduce the fundamental exact sequence of connections on a semifree DG $B$-module along a DG $B^e$-module. This result plays an essential role in our study of naïve liftability later in this paper. Again, our notation in this section comes from the previous sections.

\begin{para}\label{para20250408c}
Let $N$ be a DG $B$-module, $X$ be a DG $B^e$-module, and $D \in \SDer_A (B, X)$ be a homogeneous element. A \emph{$D$-connection on $N$} (\emph{along $X$}, to be precise) is an element $\psi\in \Hom_{\gr(A)}(N, N\otimes _B X(|D|))$ such that for all $b \in B$ and all homogeneous elements $x \in N$ the equality 
\begin{equation}\label{connection rule}
\psi (xb) = \psi (x)b + (-1)^{|D||x|}x \otimes _B  D(b)
\end{equation}
holds. Note that $|\psi|=|D|$ as an $A$-linear mapping. 
\end{para}

\begin{ex}
For a DG $B$-module $N$, the differential $\partial ^N$ is a $d^B$-connection on $N$ along $B$. 
%(In this example  $X = B$.)
\end{ex}

\begin{para}\label{0-connection}
Considering the notation from~\ref{para20250408c}, the mapping $0$ belongs to $\Der_A(B, X)_n$ for all integers $n$. Thus, by definition, an element $\psi\in \Hom_{\gr(A)}(N, N\otimes _B X(|D|))$ is a $0$-connection along $X$ if and only if $\psi$ is $B$-linear.
\end{para}

%\color{red}In the next definition, it seems that $(-1)^{|e_{\lambda_i}|}$ is needed in the display otherwise, the proof of~\ref{free case} seems to not work. Would you please check.\color{black}

\begin{para}\label{trivial connection}
Consider the notation from~\ref{para20250408c} and assume that $N$ is free as an underlying graded $B$-module with a free basis $\mathcal{B}=\{ e_{\lambda}\mid \lambda\in\Lambda\}$. 
For any derivation $D \in \Der (B, X)_{|D|}$ we define a mapping  $\varphi (D)\colon  N \to  N \otimes _B X$  by 
$$
\varphi (D) \left( \sum _{i =1} ^ r e_{\lambda_i} b_{\lambda_i} \right) = \sum _{i =1} ^ r (-1)^{|D||e_{\lambda_i}|} e_{\lambda_i} \otimes_B D(b_{\lambda_i})
$$
%$$
%\varphi (D) \left( \sum _{i =1} ^ r e_{\lambda_i} b_{\lambda_i} \right) = \sum _{i =1} ^ r \color{red}(-1)^{|e_{\lambda_i}|}\color{black}e_{\lambda_i} \otimes_B D(b_{\lambda_i})
%$$
for all sequences of elements $b_{\lambda_1}, \ldots , b_{\lambda_r} \in B$. It is straightforward to check that $\varphi (D)$ is a $D$-connection. 
We call $\varphi(D)$ the \emph{trivial $D$-connection on $N$} (\emph{along $X$}) with respect to the free basis $\mathcal{B}$.
\end{para}

Our next result lists a few properties of connections. The proofs follow directly from the definition.

\begin{prop}\label{sum of connections}
Let $N$ be a DG $B$-module and $X$ be a DG $B^e$-module. For $i=1,2$ assume that  $D_i \in \SDer _A(B, X)$  are homogeneous  and   $\psi _i\colon N \to N \otimes _B X(|D_i|)$ are $D_i$-connections on $N$  along $X$. 
Then, the following assertions hold:
\begin{enumerate}[\rm(a)]
\item\label{item20250408a}
If  $b \in B$ is homogeneous, then  $b \psi_1$ is a $bD_1$-connection. 
\item\label{item20250408b}
If  $|D_1|=|D_2|$, then  $\psi _1 \pm  \psi _2$ is a $D_1 \pm D_2$-connection. 
\item\label{item20250408c}
If $X=B$, then $[\psi _1, \psi _2]$ is a $[D_1, D_2]$-connection.   
\end{enumerate}
\end{prop}

The following result follows immediately from Proposition~\ref{sum of connections}\eqref{item20250408b} and~\ref{0-connection}.

\begin{cor}\label{Blinear}
Let $N$ be a DG $B$-module and $X$ be a DG $B^e$-module. If both $\psi_1, \psi_2\colon N \to N \otimes _BX(n)$ are $D$-connections on $N$ along $X$ of degree $n\in \mathbb{Z}$, then $\psi_1 -\psi_2$ is $B$-linear. 
\end{cor}

\begin{para}
Let $N$ be a DG $B$-module and $X$ be a DG $B^e$-module. For each integer $n$, let 
$$
\Conn (N, N \otimes _B X)_n= 
\left\{\psi\colon N\to N \otimes _B X(n)
\left| \text{\begin{tabular}{c}
$\psi$ is a $D$-connection for some \\
$D \in \Der _A(B, X)_n$
\end{tabular}}\right.\!\!\!\right\}.
$$
Moreover, we define 
$$
\SConn (N, N \otimes _B X)=  \bigoplus _{n \in \mathbb{Z}}\Conn (N, N \otimes _B X)_n. 
$$
By Proposition~\ref{sum of connections} we see that $\SConn (N, N \otimes _B X)$ has a structure of graded left $B$-module. 
Furthermore, $\SConn (N, N \otimes _B X)$ is a left DG $B$-module with the differential structure defined by the formula
$$
\partial^{\Conn}(\psi) = \partial ^{N \otimes_B X} \circ \psi - (-1) ^{|\psi |} \psi \circ \partial ^N 
$$
for all homogeneous elements $\psi \in \SConn (N, N \otimes _B X)$, where $\partial^{\Conn}(\psi)$ is the $\partial^{\Der}(D)$-connection for a $D$-connection $\psi$.

In case $X=B$, we simply denote $\SConn (N, N\otimes_BB)$ by $\SConn (N)$.
\end{para}

\begin{para}\label{zero}
Considering the notation from~\ref{para20250408c}, if $N ={0}$, then there is only one element $\psi$ in $\Hom_{\gr(A)}(N, N\otimes _B X(|D|))$, which is the zero mapping. 
Note that this $\psi$ satisfies~\eqref{connection rule} for any homogeneous derivation $D \in \SDer _A(B, X)$. 
Denote by $0_D$ the zero mapping that is regarded as a $D$-connection; we distinguish between $0_D$ and $0_{D'}$ if $D \not= D'$. 
Thus, we have the bijection
$$
\SConn (0, 0 \otimes  _B X) = \{ 0_D | \ D \in \SDer _A(B, X)\} \cong \SDer _A(B, X)
$$
\end{para}

\begin{para}
Let $N$ be a DG $B$-module and $X$ be a DG $B^e$-module. Similar to our discussion in~\ref{topology}, $\SConn (N, N \otimes _B X)$ becomes a topological DG $B$-module.
\end{para}

Again, the same argument as that of Proposition~\ref{complete} implies the following.

\begin{prop}
For a DG $B$-module $N$ and a DG $B^e$-module $X$, the topological DG $B$-module $\SConn (N, N \otimes _B X)$ is complete and separated.
\end{prop}

\begin{para}\label{para20250409a}
For a DG $B$-module $N$ that is free as an underlying graded $B$-module with a basis $\mathcal{B}=\{e_{\lambda}\mid \lambda\in \Lambda\}$ and a DG $B^e$-module $X$, there is a natural inclusion  
$$
\iota\colon \SHom _B(N, N \otimes _BX) \to \SConn (N, N \otimes _BX)  
$$
since all $B$-linear mappings $N \to N \otimes _B X$ are $0$-connections. On the other hand, there is a natural mapping 
$$
\nu\colon \SConn (N, N \otimes _BX) \to \SDer _A(B, X)
$$
which maps a $D$-connection $\psi$ to $D$. Note that $\nu$ is well-defined, that is, if $\psi\in \SConn (N, N \otimes _BX)$ is a $D$-connection, then $D$ is uniquely determined by $\psi$. More precisely, if $\psi$ is also a $D'$-connection for $D'\in \SDer _A(B, X)$, then for all $b\in B$ and the basis elements $e_{\lambda}\in \mathcal{B}$, by~\eqref{connection rule} we have $e_{\lambda}\otimes_B (D-D')(b)=0$, which implies that $D=D'$. Note also that both $\iota$ and $\nu$ are continuous.
\end{para}

\begin{para}\label{para20250408d}
Let $N$ be a DG $B$-module, $X$ be a DG $B^e$-module, and $\mathcal{L} \subseteq \SDer _A(B, X)$ be a left DG $B$-submodule. 
An \emph{$\mathcal{L}$-connection} is a DG $B$-module homomorphism
$$
\nabla\colon \mathcal{L} \to \SConn (N, N \otimes_B X).
$$
For an element $D\in \mathcal{L}$, we denote $\nabla(D)$ by $\nabla _D$. Thus, $\nabla$ is an $\mathcal{L}$-connection if and only if for all $D, D' \in \mathcal{L}$ and $b\in B$ the following conditions are satisfied: 
\begin{enumerate}[\rm(i)]
\item $\nabla _D$ is a $D$-connection on $N$ along $X$;
\item $\nabla_{D+D'} =\nabla_D+ \nabla_{D'}$;
\item $\nabla _{bD} = b \nabla _D$;    
\item $\nabla _{\partial^\Der (D)} = \partial ^{\Conn} (\nabla _D)$. 
\end{enumerate}
\end{para}

\begin{para}
Consider the setting of~\ref{para20250408d}. Assume that $X=B$ and $\mathcal{L}$ is a Lie subalgebra of $\SDer_A(B)$. Let $\nabla$ be an $\mathcal{L}$-connection. For $D_1, D_2 \in \mathcal{L}$,  we define
$$
R_{\nabla} (D_1, D_2) =  [\nabla _{D_1}, \nabla _{D_2}]  - \nabla _{[D_1, D_2]}.
$$ 
We call $R_{\nabla}$ the {\it curvature of $\nabla$}. Also, the $\mathcal{L}$-connection $\nabla$ is said to be {\it integrable} if $R_{\nabla} =0$; see~\cite{EG} for details on these topics for modules. 
\end{para}

\begin{thm}\label{fundamental}
Let $N$ be a DG $B$-module and $X$ be a DG $B^e$-module. Assume that $N$ is free as an underlying graded $B$-module (not necessarily semifree). Then, 
\begin{equation*}\label{fes} 
\E _N(X)\colon 0 \to \SHom _B(N, N \otimes _BX) \xra{\iota}  \SConn (N, N \otimes _BX) \xra{\nu} \SDer _A(B, X) \to 0
\end{equation*}
is an exact sequence in the category $\C(B)^{op}$ of left DG $B$-modules.
\end{thm}
  
\begin{proof}
The surjectivity of $\nu$ follows from~\ref{trivial connection} and the fact that  $\nu (\varphi (D)) =D$, for all $D\in \SDer _A(B, X)$. Also, by definition, $\Ker \nu$ is the set of all $0$-connections on $N$ along $X$, that is, $\Ker \nu=\SHom _B(N, N \otimes _BX)$; see~\ref{0-connection} and~\ref{para20250409a}. 
\end{proof}

\begin{para}
We call the short exact sequence $\E_N (X)$ introduced in Theorem~\ref{fundamental} the \emph{fundamental exact sequence of connections on $N$ along $X$}.

If $N=0$, then the fundamental exact sequence  $\E_0(X)$ implies an isomorphism 
$\SConn (0, 0 \otimes _BX)  \cong \SDer _A(B, X)$ that has been already remarked in~\ref{zero}. 
\end{para}

In the following result, recall the definition of free DG modules from~\ref{para20250712a}.

\begin{prop}\label{free case}
Let $N$ be a free DG $B$-module. Then, as a sequence of DG $B$-modules, $\E_N(X)$ is split for all DG $B^e$-modules $X$. If $X=B$, then $N$ admits an integrable $\SDer_A(B)$-connection. 
\end{prop}

\begin{proof}
Let $\mathcal{B}=\{ e_{\lambda}\mid \lambda\in \Lambda\}$ be a free basis for the free DG $B$-module $N$. Consider $\varphi\colon \SDer _A(B, X) \to \SConn (N, N \otimes _BX)$ which maps a derivation $D$ to the trivial $D$-connection $\varphi (D)$, introduced in~\ref{trivial connection}, with respect to the free basis $\mathcal{B}$. It is straightforward to see that $\varphi$ is a graded $B$-module homomorphism and $\nu \circ \varphi=\id_{\SDer _A(B, X)}$. 
Therefore, to prove the assertion, it only remains to show the equality
\begin{equation}\label{eq20250409a}
\partial ^{\Conn} \circ \varphi = \varphi \circ \partial ^{\Der}.
\end{equation} 
For this, let $D \in \SDer _A(B, X)$ and $b \in B$. Then,  
\begin{eqnarray*} 
 (\partial ^{\Conn} \circ \varphi )(D)(e_{\lambda}b) 
 \!\!\!\!&=&\!\!\!\! (\partial ^{\Conn} (\varphi (D)))(e_{\lambda}b) \\
 &=&\!\!\!\!\partial ^{N \otimes X} ( \varphi (D)(e_{\lambda}b) ) - (-1)^{|D|+|e_{\lambda}|}  \varphi (D) (e_{\lambda} d^Bb) \\
 &=&\!\!\!\!(-1)^{|D||e_{\lambda}|}\!\!\left(\partial ^{N \otimes X} (e_{\lambda}\otimes_B D(b)) - (-1) ^{|D|+|e_{\lambda}|} e_{\lambda}\otimes_B D(d^Bb)\right) \\
 &=&\!\!\!\!(-1)^{|D||e_{\lambda}|+|e_{\lambda}|}\left(e_{\lambda}\otimes_B\partial ^X(D(b)) - (-1)^{|D|}  e_{\lambda}\otimes_B D(d^Bb)\right) \\
 &=&\!\!\!\!(-1)^{|D||e_{\lambda}|+|e_{\lambda}|} e_{\lambda}\otimes_B\left(\partial ^X(D(b)) - (-1)^{|D|}  D(d^Bb)\right)  \\
 &=&\!\!\!\!(-1)^{|D||e_{\lambda}|+|e_{\lambda}|} e_{\lambda}\otimes_B \partial ^{\Der} (D) (b)\\
 &=&\!\!\!\!\varphi \left(\partial ^{\Der} (D)\right)(e_{\lambda}b)  \\
 &=&\!\!\!\!(\varphi \circ \partial ^{\Der} )(D) (e_{\lambda}b) 
\end{eqnarray*} 
where the second equality follows from the fact that $\partial ^N (e_{\lambda} b)= (-1)^{|e_{\lambda}|} e_{\lambda} d^Bb$. This computation establishes the equality~\eqref{eq20250409a}.

Assume now that $X=B$, and set $\nabla _D = \varphi (D)$, for an element $D \in \SDer _A(B)$. 
If  $D_1, D_2 \in \SDer_A(B)$, then by Proposition \ref{sum of connections}\eqref{item20250408c} we have $[\nabla_{D_1}, \nabla_{D_2}]$ is a $[D_1, D_2]$-connection. Also, $[\nabla_{D_1}, \nabla_{D_2}](e_{\lambda})=0$ for all of the basis elements $e_{\lambda}\in \mathcal{B}$. Hence, $[\nabla_{D_1}, \nabla_{D_2}]$ is the trivial $[D_1, D_2]$-connection, i.e., $[\nabla_{D_1}, \nabla_{D_2}]=\nabla_{[D_1,D_2]}$. This means that the curvature $R_{\nabla}$ of $\nabla$ vanishes, and therefore, $\nabla$ is integrable.
\end{proof}

\section{Naïve lifting and connections along DG $B^e$-modules}\label{sec20250421b}

Our main result in this section (and also in this paper) is Theorem~\ref{naive} in which we characterize naïve liftability of DG modules in terms of connections along DG $B^e$-modules. We start by reminding the reader of the definition of naïve liftability from~\cite{NOY1, NOY}. As usual, we still use the notation from the previous sections.

\begin{para}\label{Atiyah} 
Let $N$ be a semifree DG $B$-module with a semifree basis $\mathcal{B}=\{e_{\lambda}\mid \lambda\in \Lambda\}$. We denote by $N|_A$ the DG $B$-module $N$ regarded as a DG $A$-module via $A\to B$. Note that $N|_A$ is a semifree DG $A$-module and there is a natural DG $B$-module epimorphism $\pi_N\colon N|_A\otimes_A B\to N$ defined by $\pi_N(n\otimes_Ab)=nb$, for all $n\in N$ and $b\in B$. The DG $B$-module $N$ is called \emph{naïvely liftable} to $A$ if $\pi_N$ has a right inverse in the abelian category $\mathcal{C}(B)$, or equivalently, if the natural exact sequence 
$$
0 \to N \otimes_B J \to N \otimes_A B \xra{\pi_N} N \to 0
$$
of DG $B$-modules splits.

Note that $N$ being a semifree DG $B$-module exactly means that it is free as an underlying graded $B$-module with the free basis $\mathcal{B}$ that is indexed by a well-ordered set $\Lambda$ and for each $e_{\lambda}\in \mathcal{B}$ we have 
\begin{equation}\label{eq20250410a}
\partial ^N (e_{\lambda} ) = \sum _{\mu < \lambda} e_{\mu} b_{\mu \lambda}  
\end{equation}
where $b_{\mu \lambda} \in B$. In this setting, the $B$-linear map $\alpha\colon N \to N \otimes _B J (-1)$ defined by 
$$
\alpha (e_{\lambda} ) = \sum _{\mu < \lambda} e_{\mu } \otimes_B  \delta (b_{\mu \lambda})
$$
where $\delta$ is the universal derivation defined in~\ref{para20250410a}, is called the \emph{Atiyah homomorphism}; see~\cite{NOY3} for more details. The map $\alpha$ is, in fact, a DG $B$-module homomorphism of degree $-1$ that fits into an exact triangle 
$$
N\otimes_BJ \to N|_A \otimes _AB \xra{\pi_N} N \xra{\alpha} N \otimes _B J (-1)
$$
in the homotopy category $\K (B)$. It is shown in~\cite[Proposition 4.10]{NOY3} that $N$ is naïvely liftable to $A$ if and only if $\alpha$ is null homotopic. 
\end{para}

Another description of the Atiyah homomorphism using connections is as follows; compare to~\cite[Theorem 4.9]{NOY3}. 

\begin{prop}\label{atiyah}
For a semifree DG $B$-module $N$ we have the equality
$$
\alpha = - \partial ^{\Conn} (\varphi (\delta)).
$$
\end{prop}

\begin{proof}
Let $\mathcal{B}=\{e_{\lambda}\mid \lambda\in \Lambda\}$ be a semifree basis for the DG $B$-module $N$. Note that $\partial^{\Conn}(\varphi(\delta))$ is a $\partial^{\Der}(\delta)$-connection, that is, $\partial^{\Conn}(\varphi(\delta))$ is $B$-linear. Then, by definition, for each $e_{\lambda}\in \mathcal{B}$ we have the equalities
\begin{eqnarray*}
-\partial^{\Conn}(\varphi (\delta))(e_{\lambda}) &= &
\varphi(\delta)(\partial^N(e_{\lambda}))-\partial^{N\otimes_B J} (\varphi(\delta)(e_{\lambda})) \\
&=& \sum_{\mu < \lambda} e_{\mu} \otimes_B \delta(b_{\mu \lambda}) \\ 
&=&\alpha(e_{\lambda})
\end{eqnarray*}
in which the second equality follows from~\eqref{eq20250410a} and the fact that $\varphi(\delta)(e_{\lambda})=0$.
\end{proof}

\begin{thm}\label{naive}
For a semifree DG $B$-module $N$ the following are equivalent:
\begin{enumerate}[\rm(i)]
\item  $N$ is naïvely liftable to $A$; 
\item There exists a $\delta$-connection  $\psi\in \SConn (N, N \otimes _BJ)$ such that $\partial ^{\Conn} (\psi)=0$; 
\item $\E_N(J)$ splits as a sequence of DG $B$-modules;
\item $\E_N(X)$ splits as a sequence of DG $B$-modules for all DG $B^e$-modules $X$; 
\item $N$ admits a $\SDer _A(B, J)$-connection; 
\item $N$ admits a $\SDer _A(B, X)$-connection for all DG $B^e$-modules $X$.  
\end{enumerate}
\end{thm}

\begin{proof}
(ii)$\implies$(i):  
Assume that there exists a $\delta$-connection $\psi\in \SConn (N, N \otimes _BJ)$ such that $\partial ^{\Conn} (\psi)=0$. Then, $f= \psi - \varphi  (\delta)$  is $0$-connection by Proposition~\ref{sum of connections}\eqref{item20250408b}, and hence, $f$ is $B$-linear, that is, $f\in \SHom_B(N,N\otimes_BJ)$. Thus, 
$$
\partial ^{\Hom } (f) = \partial ^{\Conn} (\iota (f) ) =  - \partial^{\Conn} (\varphi (\delta))= \alpha 
$$
where the last equality comes from Proposition~\ref{atiyah}. Therefore, $\alpha$ is null homotopic. 

(i)$\implies$(ii):  
Assume that $N$ is naïvely liftable to $A$. By our discussion in~\ref{Atiyah}, the map $\alpha$ is null homotopic. 
Hence, there is a graded $B$-linear mapping $f\colon N \to N \otimes _B J$ such that  $\partial ^{\Conn} (\iota(f)) = \partial ^{\Hom } (f)=\alpha$. Set $\psi = f + \varphi (\delta)$. Using Proposition~\ref{atiyah}, we can check that $\psi$ satisfies the requirements in (ii). 

(iii)$\implies$(ii):  
Assume that $\E_N(J)$ splits as a sequence of DG $B$-modules. Hence, $\SDer _A(B, J)$ is isomorphic to a direct summand of $\SConn (N, N \otimes _B J)$ as a DG $B$-module. 
Since, as we mentioned in~\ref{para20250410a}, $\delta \in \SDer _A(B, J)$ is a cycle, there is a corresponding cycle $\psi$ in $\SConn (N, N \otimes _B J)$ which satisfies the requirements in (ii). 
 
(iv)$\Longleftrightarrow$(vi) (and (iii)$\Longleftrightarrow$(v)): 
Assume that, for all DG $B^e$-modules $X$, the short exact sequence $\E_N(X)$ of DG $B$-modules splits. Thus, $\nu$ is a split epimorphism, i.e., there is a DG $B$-module homomorphism $\Psi\colon \SDer _A(B, X) \to \SConn (N, N \otimes _BX)$ such that $\nu \circ \Psi = \id_{\SDer _A(B, X)}$. 
We define the $\SDer _A(B, X)$-connection $\nabla$ by setting $\nabla _D = \Psi(D)$, for all $D\in \SDer _A(B, X)$. 

Conversely, assume that for each DG $B^e$-modules $X$, the DG $B$-module $N$ admits a $\SDer _A(B, X)$-connection $\nabla$. Then, by setting $\Psi(D) = \nabla _D$, for all  $D \in \SDer _A(B, X)$, we have $\nu \circ \Psi = \id_{\SDer _A(B, X)}$, as desired.    
 
(ii)$\implies$(iv):  
Let $X$ be DG $B^e$-module, and assume that there exists a $\delta$-connection  $\psi\in \SConn (N, N \otimes _BJ)$ such that $\partial ^{\Conn} (\psi)=0$. Considering $\varpi$ from~\eqref{eq20250410b}, note that $\varpi^{-1}$ commutes with differentials, and thus, for all elements $D \in \SDer _A (B, X)$ we have the equality 
\begin{equation}\label{eq20250410d}
\varpi^{-1}(\partial ^{\Der} (D) ) = \partial ^{\Hom} (\varpi^{-1}(D)).
\end{equation}
Now, we define a map
$$
\Psi\colon \SDer _A(B, X) \to \SConn (N, N\otimes _B X)
$$  
by the formula
$$
\Psi (D) = (\id_N  \otimes_B \varpi^{-1}(D)) \circ \psi  
$$
for all homogeneous elements $D \in \SDer _A (B, X)$. Note that $|\Psi|=|\psi|$ and $\Psi (D)$ is a $D$-connection on $N$ along $X$. In fact, for all $b\in B$ and all homogeneous elements $x\in N$ we have the equalities
\begin{eqnarray*}
\Psi(D) (xb) &=& (\id_N  \otimes_B \varpi^{-1}(D))(\psi (xb)) \\ 
&=&  (\id_N  \otimes_B \varpi^{-1}(D))(\psi(x)b + x\otimes_B \delta(b)) \\
&=& \Psi(D)(x)b + (-1)^{|x||D|} x \otimes_B D(b)
\end{eqnarray*}
where the last equality comes from~\eqref{eq20250410c}. Hence, by definition of $\nu$ we have  $(\nu \circ \Psi)(D)=D$, that is, $\nu \circ \Psi=\id_{\SDer _A (B, X)}$. 
It is straightforward to see that $\Psi$ is a graded $B$-module homomorphism. It now remains to show that $\partial^{\Conn} \circ \Psi = \Psi \circ \partial ^{\Der}$, i.e., $\Psi$ is a DG $B$-module homomorphism.   
To see this equality, let $D \in \SDer _A (B, X)$ be a homogeneous element. Then, we have the equalities
\begin{eqnarray*} 
(\partial ^{\Conn} \circ \Psi  )(D)\!\!\!\!&=&\!\!\!\! \partial ^{\Conn} (\Psi (D) )  \\
 &=&\!\!\!\!\partial ^{N \otimes_B X} \circ  \Psi  (D)  - (-1)^{|\Psi (D)|} \Psi (D) \circ \partial^N   \\
 &=&\!\!\!\!\partial ^{N \otimes_B X} \circ  (\id_N\otimes_B \varpi^{-1}(D))\circ \psi  - (-1) ^{|\psi| +|D|}(\id_N\otimes_B \varpi^{-1}(D))\circ \psi \circ \partial ^N \\
 &=&\!\!\!\!\partial ^{N \otimes_B X} \circ  (\id_N\otimes_B \varpi^{-1}(D))\circ \psi  - (-1)^{|D|} (\id_N\otimes_B \varpi^{-1}(D))\circ \partial ^{N\otimes_B J} \circ \psi  \\
 &=&\!\!\!\!    \left[\id_N\otimes_B (\partial ^X \circ \varpi^{-1}(D))  + (\partial ^N \otimes_B \id_X) (\id_N\otimes_B \varpi^{-1}(D) ) \right] \circ \psi 
  \\ &&\!\!\!\! -(-1)^{|D|}  \left[ \left(\id_N\otimes_B \varpi^{-1}(D)\right)\left(\partial ^{N} \otimes_B \id_X\right) + \id_N\otimes_B (\varpi^{-1}(D) \circ \partial ^{J}) \right] \circ \psi   \\
 &=&\!\!\!\!    \left(\id_N\otimes_B (\partial ^X \circ \varpi^{-1}(D)\right) \circ \psi  - (-1)^{|D|}  \left(\id_N\otimes_B (\varpi^{-1}(D) \circ \partial ^{J})\right) \circ \psi   \\
 &=&\!\!\!\!   \left(\id_N\otimes_B \partial ^{\Hom} (\varpi^{-1}(D)\right)  \circ \psi   \\
 &=&\!\!\!\!   \left(\id_N\otimes_B \varpi^{-1}(\partial ^{\Der} (D) \right)\circ \psi  \\
 &=&\!\!\!\!  \left(\Psi  \circ \partial ^{\Der} \right)(D)
\end{eqnarray*}
where the third equality follows from the assumption that $\partial ^{\Conn} (\psi)=0$, the sixth equality follows from the fact that
$$
(\partial ^N \otimes_B \id_X) (\id_N\otimes_B \varpi^{-1}(D) )=(-1)^{|D|}\left(\id_N\otimes_B \varpi^{-1}(D)\right)\left(\partial ^{N} \otimes_B \id_X\right)
$$
and finally, the eighths equality comes from~\eqref{eq20250410d}. This computation completes the proof. 
\end{proof}

\begin{cor}
Let $N$ be a semifree DG $B$-module. If the homology mapping $$\HH_0(\nu)\colon \HH_0(\SConn (N, N \otimes _BJ)) \to \HH_0(\SDer _A(B, J))$$ 
induced by $\nu$ from~\ref{para20250409a} is surjective, then $N$ is naïvely liftable to $A$.
\end{cor}

\begin{proof}
As we mentioned in~\ref{para20250410a}, the universal derivation $\delta$ is a cycle of degree $0$ in $\SDer _A(B, J)$. 
Since  $\nu$  is a surjective DG $B$-module homomorphism, it induces the surjection on the boundary sets. 
Thus, if $\HH_0(\nu)$ is a surjection, then it induces a surjective mapping 
$Z_0(\SConn (N, N \otimes _BJ)) \to Z_0(\SDer _A(B, J))$ between the cycle sets as well. Hence, $\delta$ is the image of a cycle element of $\SConn (N, N \otimes _BJ)$ under $\nu$.  
Therefore, condition (ii) in Theorem~\ref{naive} is satisfied, which is equivalent to $N$ being naïvely liftable to $A$. 
\end{proof}
 
\section{Naïve lifting and connections along DG $B$-modules}\label{sec20250421d}

Theorem~\ref{naive} characterizes naïve liftability with respect to connections along DG $B^e$-modules, and in particular, along the diagonal ideal $J$. In this section, we pay a closer attention to connections along DG $B$-modules considered as DG $B^e$-modules via the isomorphism $B\cong B^e/J$, including $B$ itself. As we mention at the end of this section, unlike Theorem~\ref{naive}, in this situation we do not know whether naïve liftability can be characterized with respect to connections along DG $B$-modules. However, our main result in this section, namely Theorem~\ref{fesOX}, introduces several equivalent conditions some of which are similar to the ones in Theorem~\ref{naive} and an additional condition that is related to the Kodaira-Spencer homomorphism. As a final reminder, we still use the notation that we established in the previous sections.

%\begin{para}
%For a DG $B$-module $N$, there are fundamental exact sequences $\E_N(B)$ and $\E_N(\Omega)$ of connections on $N$ along $B$ and $\Omega$, respectively. More precisely, we have the short exact sequences 
%\begin{gather*} 
%\!\!\!\!\!\!\!\!\!\!\!\!\!\!\!\!\!\!\!\!\!\!\!\!\!\!\!\!\!\!\!\!\!\!\!\!\!\!\!\!\!\!\!\!\!\!\!\!\!\!\!\!\!\!\!\!\!\!\E_N(B): 0 \to \SEnd_B(N) \xra{\iota} \SConn (N) \xra{\nu} \SDer_A(B)  \to 0\\
%\E_N (\Omega): 0 \to \SHom _B(N, N \otimes _B \Omega) \xra{\iota} \SConn (N, N \otimes _B\Omega)  \xra{\nu} \SDer _A(B, \Omega)  \to 0.
%\end{gather*}
%Let $X$ be another DG $B$-module, i.e., $X$ is a DG $B^e$-module that is annihilated by $J$. Then, we also have the fundamental exact sequence $\E_N(X)$ of connections of $N$ along $X$. 

%The \emph{classical Atiyah map} $\bar\alpha\colon N \to N \otimes _B \Omega(-1)$ is defined to be a composition of the Atiyah homomorphism $\alpha\colon N \to N \otimes_B J(-1)$ and the natural projection $J \to \Omega$. 
%\end{para}
 
\begin{para}
The \emph{classical Atiyah map} $\bar\alpha\colon N \to N \otimes _B \Omega(-1)$ is defined to be the composition $(\id_N\otimes_B \pi)\circ \alpha$, where $\alpha\colon N \to N \otimes_B J(-1)$ is the Atiyah homomorphism introduced in~\ref{Atiyah} and $\id_N\otimes_B \pi\colon N\otimes_BJ(-1)\to N\otimes_B \Omega$ is the map induced by the natural projection $\pi\colon J \to \Omega=J/J^2$. 
\end{para}

\begin{para}
Let $N$ be a semifree DG $B$-module and $\bar\delta\colon B \to \Omega=J/J^2$ be the natural map induced by the universal derivation $\delta\colon B\to J$. Then, one can define a trivial $\bar\delta$-connection $\varphi ( \bar\delta)\colon N \to N \otimes _B \Omega$ with respect to a certain free basis of $N$. Then, similar to Proposition~\ref{atiyah}, we can see that $\bar\alpha = - \partial ^{\Conn} (\varphi (\bar\delta))$.
\end{para}

\begin{para}
Let $N$ be a semifree DG $B$-module with a semifree basis $\mathcal{B}=\{e_{\lambda}\mid \lambda\in \Lambda\}$. For an $A$-derivation $D  \in \SDer _A(B)$, the trivial $D$-connection $\varphi (D)$ on $N$ along $B$, discussed in~\ref{trivial connection}, can be identified by the map  $\varphi (D)\colon N \to N$ defined as $\varphi ( e_{\lambda} b )= e_{\lambda}D(b)$, for all $b \in B$ and $\e_{\lambda}\in \mathcal{B}$.
In this setting, we define the mapping  
$$
\kappa\colon \SDer _A (B) \to \SEnd _B(N) (-1)
$$
by the formula
\begin{equation}
\kappa (D) = [ \partial ^N, \varphi (D)] - \varphi ( [d^B, D]) 
\end{equation}
for all homogeneous elements $D \in \SDer _A (B)$. We call $\kappa$ the \emph{Kodaira-Spencer homomorphism on $N$}. 
Note that, by Proposition~\ref{sum of connections}\eqref{item20250408c}, for any homogeneous element $D \in \SDer _A (B)$,  both  $[\partial ^N, \varphi (D)]$ and $\varphi ([d^B, D]) $ are $[d^B, D]$-connections on $N$ of the same degree $|D|-1$. Hence, $\kappa (D)\colon N \to N(|D|-1)$ is a $B$-linear map by Corollary~\ref{Blinear} and therefore, it is well-defined. Furthermore, it is straightforward to see that $\kappa$ is a DG $B$-module homomorphism. 
\end{para}

\begin{para}\label{para20250410h}
We say that the DG $R$-algebra homomorphism $A \to B$ \emph{has semifree differential module} if $\Omega=J/J^2$ is a semifree DG $B$-module, i.e., if there is a semifree basis $\mathcal{B}=\{u_{\lambda}\mid \lambda\in \Lambda\}$ indexed by a well-ordered set $\Lambda$ such that $\{\lambda \mid |u_{\lambda}| = n\}$ is a finite set for all integers $n$ and that
$\partial ^\Omega (u_{\lambda})$ belongs to $\sum _{\mu<\lambda} u_{\mu}B$ for all $\lambda\in \Lambda$. In this setting, $\Omega = \bigoplus _{\lambda\in\Lambda} u_{\lambda} B$ is the underlying graded free $B$-module and for each $u_{\lambda}\in \mathcal{B}$ there exist $c_{\mu\lambda} \in B$ such that
$$
\partial ^{\Omega} (u_{\lambda}) = \sum _{\mu<\lambda} u_{\mu} c_{\mu\lambda}.
$$

Note that both of the DG algebra extensions discussed in Example~\ref{important} have semifree differential module; see~\cite{NOY1}. In fact, many of quasi-smooth DG algebra extensions in the sense of~\cite{NOY1} satisfy this condition. In this sense, having a semifree differential module is nearly equivalent to being quasi-smooth.   

\end{para}

The main result of this section is the following theorem.

\begin{thm}\label{fesOX}
For a semifree DG $B$-module $N$ the following are equivalent: 
\begin{enumerate}[\rm(i)]
\item
The classical Atiyah map $\bar\alpha$ is null homotopic; 
\item 
There is a $\bar\delta$-connection  $\psi \in \SConn (N, N \otimes _B \Omega)$ such that $\partial ^{\Conn}(\psi)=0$;   
\item
$\E_N(\Omega)$ splits as a sequence of DG $B$-modules; 
\item
$\E_N(X)$ splits as a sequence of DG $B$-modules for all DG $B$-modules $X$; 
\item 
$N$ admits a $\SDer _A(B, \Omega)$-connection;
\item 
$N$ admits a $\SDer _A(B, X)$-connection for all DG $B$-modules $X$.\vspace{2mm}

\noindent \!\!\!\!\!\!\!\!\!\!\!\!\!\!Furthermore, if $N$ is non-negatively graded and $A \to B$ has semifree differential \

\noindent \!\!\!\!\!\!\!\!\!\!\!\!\!\!\!\! module, then
the above conditions are equivalent to the following:\vspace{2mm}

\item $\E_N(B)$ splits as a sequence of DG $B$-modules;
\item  $N$ admits a $\SDer_A(B)$-connection;
\item The Kodaira-Spencer homomorphism $\kappa$ is null homotopic.
\end{enumerate}
\end{thm}

The proof of this theorem will be given after the following preparations.

\begin{para}\label{para20250410v}
Assume that $A \to B$ has semifree differential and consider the setting of~\ref{para20250410h}. Since, from~\eqref{eq20250410e}, there is an isomorphism $g\colon  \SDer _A(B ) \xra{\cong}  {}^*\!\Hom _B (\Omega, B)$ of DG $B$-modules, 
we find  $\partial _{\lambda} \in \SDer _A(B)$ such that  $|\partial_{\lambda}| = - |u_{\lambda}| <0$ and 
\begin{equation}\label{guj}
g(\partial _{\lambda})(u_{\mu}) = \begin{cases}
1 &  (\lambda=\mu) \\ 0 & (\lambda \not= \mu).  \end{cases}
\end{equation}
\end{para}

\begin{lem}\label{sublemma}
Considering the setting of~\ref{para20250410h} and~\ref{para20250410v}, the following assertions hold: 
\begin{enumerate}[\rm(a)]
\item 
For all $\lambda\in \Lambda$ we have the equality
\begin{equation}\label{eq20250410f}
\partial^{\Der} (\partial _{\lambda} ) =(-1)^{|u_{\lambda}|+1} \sum _{\upsilon > \lambda}  c_{\lambda \upsilon}\partial _{\upsilon}.
\end{equation}
\item
For all $b \in B$ we have the equality 
\begin{equation}\label{eq20250410g}
\bar\delta (b) = \sum _{\lambda}  u_{\lambda} \partial_{\lambda} (b).
\end{equation}
\end{enumerate}
\end{lem}

The right-hand side of~\eqref{eq20250410f} is a power series, and hence, it is an infinite sum. On the other hand, the right-hand side of~\eqref{eq20250410g} is a finite sum.

\begin{proof}
(a) By definition, for all $\lambda, \mu\in \Lambda$, we have the equalities
\begin{eqnarray*}
 g(\partial^{\Der} (\partial _{\lambda} ))(u_{\mu}) &=&\left( \partial ^{\Hom} (g(\partial _{\lambda}))\right)(u_{\mu}) \\ 
&=& d^B \left(g(\partial _{\lambda})(u_{\mu})\right) - (-1)^{|\partial_{\lambda}|}  g(\partial _{\lambda}) (\partial ^\Omega (u_{\mu}))  \\
&=& (-1)^{|\partial_{\lambda}|+1}g(\partial _{\lambda})\left(\sum _{\upsilon<\mu} u_{\upsilon} c_{\upsilon\mu} \right)\\  
&=& (-1)^{|u_{\lambda}|+1} \sum _{\upsilon<\mu} \left(g(\partial _{\lambda}\right)(u_{\upsilon}) )c_{\upsilon\mu} \\
&=& \begin{cases} (-1)^{|u_{\lambda}|+1} c_{\lambda\mu} & \lambda < \mu \\ 0 & \lambda \geq \mu.  \end{cases}
\end{eqnarray*} 
On the other hand, we have the equalities
\begin{eqnarray*} 
 g\left(\sum _{\upsilon > \lambda} (-1)^{|u_{\lambda}|+1} c_{\lambda\upsilon}\partial _{\upsilon}\right)(u_{\mu}) 
&=& \sum _{\upsilon > \lambda} (-1)^{|u_{\lambda}|+1} c_{\lambda\upsilon}g(\partial _{\upsilon}) (u_{\mu} ) \\ 
&=&  \begin{cases} (-1)^{|u_{\lambda}|+1} c_{\lambda\mu} & \lambda < \mu \\ 0 & \lambda \geq \mu  \end{cases}
\end{eqnarray*} 
where the first equality comes from the fact that $g$ is continuous. Hence, 
$$
g\left(\partial^{\Der} (\partial _{\lambda} )\right)=g\left(\sum _{\upsilon > \lambda} (-1)^{|u_{\lambda}|+1} c_{\lambda\upsilon}\partial _{\upsilon}\right)
$$
which implies the desired equality~\eqref{eq20250410f}. 

(b) Note that the isomorphism $g$ is given by the equality 
$g(D) \circ \bar\delta = D$, for all $D \in \SDer _A(B)$. 
In particular, for any $\mu\in \Lambda$, we have $g(\partial_{\mu}) (\bar\delta (b)) = \partial _{\mu} (b)$. 
Since  $\Omega$ is generated by $\mathcal{B}$ as an underlying graded $B$-module, we can write  $\bar\delta (b) = \sum _{\lambda}   u_{\lambda} b_{\lambda} $ for $b_{\lambda} \in B$. 
It then follows from~\eqref{guj} that $b_{\mu} = g(\partial_{\mu})(\bar\delta (b)) = \partial _{\mu} (b)$, which implies the equality~\eqref{eq20250410g}, as desired.  
\end{proof}

\noindent \emph{Proof of Theorem~\ref{fesOX}.}
The equivalences of (i)-(vi) are shown by the same arguments as in the proof of Theorem~\ref{naive}. Also, the implication (iv)$\implies$(vii) is trivial and (vii)$\Longleftrightarrow$(viii) is proven similar to the equivalence (iv)$\Longleftrightarrow$(vi).

(vii)$\implies$(ii): Assume that $\E_N(B)$ splits as a sequence of DG $B$-modules. Consider the setting of~\ref{para20250410h} and~\ref{para20250410v}. Since  $\nu\colon \SConn (N) \to \SDer _A(B)$ is a split epimorphism, there is a DG $B$-module homomorphism $\Psi\colon \SDer _A(B) \to \SConn (N)$ such that $\nu \circ \Psi = \id_{\SDer _A(B)}$. For each $\lambda\in\Lambda$, set $\psi _{\lambda}= \Psi (\partial _{\lambda})$ and note that $\psi _{\lambda} \in \SConn (N)$ is a $\partial _{\lambda}$-connection with  $| \psi _{\lambda}| = |\partial_{\lambda}|=  -|u_{\lambda}|$. 
Now, we define a map $\psi\colon N \to N \otimes _B \Omega$ by the formula
\begin{equation}\label{eq20250411a}
\psi (x) = \sum _{\lambda} (-1)^{(|x|+ |u_{\lambda}|)|u_{\lambda}|}\psi _{\lambda} (x) \otimes_B u_{\lambda} 
\end{equation}
for all homogeneous elements $x \in N$ and $u_{\lambda}\in\mathcal{B}$. Note that the sum on the right-hand side of~\eqref{eq20250411a} is a finite sum because $\psi _{\lambda} (x) = 0$ if $|x| < |u_{\lambda}| = -|\partial _{\lambda}|$ and there are only finitely many $\lambda$ satisfying the inequality $|u_{\lambda}| \leq |x|$. We show that $\psi$  satisfies condition (ii). For all homogeneous elements $x\in N$ and $u_{\lambda}\in\mathcal{B}$ we have 
\begin{eqnarray*} 
\psi (xb)\!\!\!\! &= &\!\!\!\! \sum _{\lambda} (-1)^{(|x|+|b|+ |u_{\lambda}|)|u_{\lambda}|}\psi _{\lambda} (xb) \otimes_B u_{\lambda}  \\
&=&\!\!\!\! \sum _{\lambda}  (-1)^{(|x|+|b|+ |u_{\lambda}|)|u_{\lambda}|}\left( \psi _{\lambda} (x) b\otimes_B u_{\lambda} + (-1)^{|x||\partial _{\lambda}|}x \partial _{\lambda}(b) \otimes_B u_{\lambda}  \right)\\ 
&=&\!\!\!\! \sum _{\lambda}  (-1)^{(|x|+|b|+ |u_{\lambda}|)|u_{\lambda}|}\psi _{\lambda} (x) \otimes_B bu_{\lambda} +   x \otimes_B \sum _{\lambda} (-1)^{(|b|+ |u_{\lambda}|)|u _{\lambda}|}\partial _{\lambda}(b)u_{\lambda}  \\
&=&\!\!\!\! \sum _{\lambda}   (-1)^{(|x|+ |u_{\lambda}|)|u_{\lambda}|} \psi _{\lambda} (x) \otimes_B  u_{\lambda} b +   x \otimes_B \sum _{\lambda} u_{\lambda} \partial _{\lambda}(b) \\
&=&\!\!\!\!  \psi (x)b + x \otimes_B \bar\delta (b). 
\end{eqnarray*}
where the last equality follows from~\eqref{eq20250410g}. Therefore, $\psi$ is a $\bar\delta$-connection. It now remains to prove $\partial ^{\Conn} (\psi) = 0$.
For this purpose, again, let $x \in N$ be a homogeneous element and $u_{\lambda}\in \mathcal{B}$.  Then, we have the equalities 
\begin{eqnarray}
(\partial ^{\Conn} (\psi ))(x)\!\!\!\!  &=&\!\!\!\! \left(\partial ^{N \otimes_B \Omega} \circ  \psi  - \psi \circ \partial^N \right) (x)\notag  \\
&=&\!\!\!\! \partial ^{N \otimes_B \Omega} \left( \sum _{\lambda} (-1)^{(|x|+|u_{\lambda}|)|u_{\lambda}|}\psi _{\lambda} (x) \otimes_B u_{\lambda} \right)\notag\\ 
&&- \sum _{\lambda} (-1)^{(|x|-1+ |u_{\lambda}|)|u_{\lambda}|}\psi _{\lambda} (\partial^N(x))  \otimes_B u_{\lambda}\notag  \\ 
&=&\!\!\!\! 
\sum _{\lambda} (-1)^{(|x|+|u_{\lambda}|)|u_{\lambda}|}\!\! \left( \partial ^N(\psi_{\lambda}(x)) \otimes_B u_{\lambda} + (-1) ^{|x|+|u_{\lambda}|} \psi _{\lambda} (x) \otimes_B \partial ^{\Omega}(u_{\lambda})\right)\notag\\
&&-\sum _{\lambda} (-1)^{(|x|-1+ |u_{\lambda}|)|u_{\lambda}|}\psi _{\lambda} (\partial^N(x))  \otimes_B u_{\lambda}\notag \\ 
&=&\!\!\!\!  
\sum _{\lambda} (-1)^{(|x|+|u_{\lambda}|)|u_{\lambda}|} \left( \partial ^N(\psi_{\lambda}(x)) - (-1)^{|u_{\lambda}|}\psi _{\lambda} (\partial^N(x)) \right) \otimes_B u_{\lambda} \notag\\
&&+  \sum _{\lambda} (-1)^{(|x|+|u_{\lambda}|)|u_{\lambda}|+|x| +|u_{\lambda}|}  \psi _{\lambda} (x) \otimes_B \partial ^{\Omega}(u_{\lambda})\notag \\ 
&=&\!\!\!\!  
\sum _{\lambda} (-1)^{(|x|+|u_{\lambda}|)|u_{\lambda}|} \partial^{\Conn} (\psi_{\lambda}) (x) \otimes_B u_{\lambda}\notag \\
&&+  \sum _{\lambda} (-1)^{(|x|+|u_{\lambda}|)|u_{\lambda}|+|x| +|u_{\lambda}|}  \psi _{\lambda} (x) \otimes_B \partial ^{\Omega}(u_{\lambda}).\label{eq20250411s}
\end{eqnarray}  
For each $\lambda\in \Lambda$, the equalities 
\begin{eqnarray} 
\partial^{\Conn} (\psi_{\lambda}) &=& \partial^{\Conn} (\Psi(\partial_{\lambda}))\notag\\
&=& \Psi (\partial ^{\Der} (\partial_{\lambda}))\notag \\
&=& \Psi \left((-1)^{|u_{\lambda}|+1}\sum _{\upsilon>\lambda} c_{\lambda\upsilon} \partial _{\upsilon} \right)\notag\\
&=&  (-1)^{|u_{\lambda}|+1}\sum _{\upsilon>\lambda} c_{\lambda\upsilon} \psi _{\upsilon}\label{eq20250411x}  
\end{eqnarray}  
hold. Also, since $|c_{\lambda\mu}|= |u_{\mu}|-|u_{\lambda}| -1$ for all $\lambda, \mu\in \Lambda$, we have
\begin{eqnarray}
\psi _{\lambda} (x) \otimes_B \partial ^{\Omega}(u_{\lambda}) \!&=&\!\! \sum _{\nu<\lambda} \psi _{\lambda}(x) \otimes_B u_{\nu} c_{\nu\lambda}\notag\\
&=&\!\! \sum _{\nu<\lambda} (-1) ^{(|u_{\nu}|+|x|+1)(|u_{\lambda}|-|u_{\nu}|-1)} c_{\nu\lambda} \psi _{\lambda} (x) \otimes_B u_{\nu}.\label{eq20250411c}
\end{eqnarray}
Substituting~\eqref{eq20250411x} and~\eqref{eq20250411c} in the last equality of~\eqref{eq20250411s}, one can see that $(\partial ^{\Conn} (\psi)) (x)= 0$, as desired.

(ix)$\implies$(vii): If the Kodaira-Spencer homomorphism $\kappa$ is null homotopic, then there is a graded $B$-module homomorphism $h\colon \SDer _A (B) \to \SEnd _B(N)$ such that $\kappa (D) = -h( \partial ^{\Der} (D) ) +\partial ^{\End}(h(D))$, for all $D \in \SDer _A(B)$.  This means that
$$
[\partial ^N, \varphi (D)] - \varphi ([d^B, D]) =  [\partial ^N,  h (D)] - h ([d^B, D]).   
$$
Therefore, setting  $\psi (D) = (\varphi -h)(D)$, we see that  $\psi (D)$ is a $D$-connection on $N$ along $B$ satisfying the equality $ [\partial ^N , \psi (D)] = \psi ([d^B, D]) $. 
It then follows that $\psi\colon \SDer _A (B) \to \SConn (N)$  is a DG $B$-module homomorphism such that $\nu \circ \psi = \id_{\SDer _A (B)}$, and hence, $\E_N(B)$ splits.

(vii)$\implies$(ix): This implication is proved by following the steps of the proof of (ix)$\implies$(vii) in reverse.
%(i)$\Longleftrightarrow$(x): Note that, the classical Atiyah map $\bar\alpha$ is null homotopic if and only if the Atiyah homomorphism $\alpha$ is null homotopic. Hence, the assertion follows from our discussion in~\ref{Atiyah}. This concludes the proof of %the theorem.
\qed
\vspace{2mm}

Unlike Theorem~\ref{naive}, we do not know whether naïve liftability can be characterized by the conditions of Theorem~\ref{fesOX}. Hence, we conclude this paper with the following question:

\begin{question}
Under the assumptions of Theorem~\ref{fesOX}, are conditions (i)-(ix) equivalent to $N$ being naïvely liftable to $A$?
\end{question}

%%%%%%%%%%%%%%%%%%
%%%%%%%%%%%%%%%%%%
%%%%%%%%%%%%%%%%%%

%\appendix
%\section{}\label{appendixB}

\section*{Acknowledgments}
The second and third authors are grateful to the Department of Mathematical Sciences
at Georgia Southern University for the hospitality during their visit in March 2024 when a part of this study was carried out.

%\bibliography{../+new}

\begin{thebibliography}{10}
\bibitem{auslander:lawlom}
M. Auslander, S.\ Ding, and \O.\ Solberg, \emph{Liftings and weak liftings of
  modules}, J. Algebra, \textbf{156} (1993), 273--317.

\bibitem{avramov:ifr}
L. L. Avramov, \emph{Infinite free resolutions}, Six lectures on commutative
  algebra (Bellaterra, 1996), Progr. Math., vol. 166, Birkh\"auser, Basel,
  1998, pp.~1--118.
  
\bibitem{avramov:dgha}
L. L. Avramov, H.-B.\ Foxby, and S.\ Halperin, \emph{Differential graded
  homological algebra}, in preparation.
  
\bibitem{AINSW}
L. L. Avramov, S. B. Iyengar, S. Nasseh, K. Sather-Wagstaff,
\textit{Homology over trivial extensions of commutative DG algebras},  Comm. Algebra, \textbf{47} (2019), 2341--2356.

\bibitem{cartan2}
E. Cartan, \textit{Sur les variétés à connexion projective}, Bull. Soc. Math. France, \textbf{52} (1924), 205--241.

\bibitem{cartan1}
E. Cartan, \textit{Les espaces à connexion conforme}, Ann. Soc. Polon. Math., \textbf{2} (1923), 171--221.

\bibitem{christoffel}
E. B. Christoffel, {\it Über die Transformation der homogenen Differentialausdrücke zweiten Grades. Von Herrn E. B. Christoffel in Zürich}, J. für die reine und angew. Math., \textbf{70} (1869), 46--70.

\bibitem{AC}
A. Connes, {\it Noncommutative differential geometry}, Inst. Hautes \'{E}tudes Sci. Publ. Math., {\bf 62} (1985), 257--360.

\bibitem{CQ}
J. Cuntz and D. Quillen, \emph{Algebra extensions and nonsingularity}, J. Amer. Math. Soc., \textbf{8} (1995), no.~2, 251--289.

\bibitem{ehresmann1}
C. Ehresmann, \emph{Les connexions infinitésimales dans un espace fibré diffrentiable}, Colloque de Toplogie, CBRM, Bruxelles, pp. 29--55, (1950).

\bibitem{ehresmann2}
C. Ehresmann, \emph{Sur les espaces fibrés associés à une variété différentiable}, C. R. Acad. Sc. t., {\bf 216} (1943), 628--630.

\bibitem{E}
E. Eriksen, \emph{Connections on modules over quasi-homogeneous plane curves}, Comm. Algebra, {\bf 36} (2008), no. 8, 3032--3041.

\bibitem{EG}
E. Eriksen and T. S. Gustavsen, \emph{Connections on modules over singularities of finite and tame CM representation type}, Springer-Verlag, Berlin, 2009, pp.~99--108.

\bibitem{felix:rht}
Y.\ F{\'e}lix, S.\ Halperin, and J.-C.\ Thomas, \emph{Rational homotopy
  theory}, Graduate Texts in Mathematics, vol. 205, Springer-Verlag, New York,
  2001.

\bibitem{GL}
Tor H. Gulliksen and G. Levin,
{\it Homology of local rings}, Queen's Paper in Pure and Applied Mathematics, No. 20 (1969), Queen's University, Kingston, Ontario, Canada.

\bibitem{koszul}
J. L. Koszul, \emph{Homologie et cohomologie des algebres de Lie}, Bull. Soc. Math. France, {\bf 78} (1950), 65--127.

%\bibitem{levi}
%U. Krähmer, \emph{Dirac Operators, Lecture 1: Projective Modules and Connections}, IPM Tehran 19, (2009).

\bibitem{levi}
T. Levi-Civita and G. Ricci, \emph{Méthodes de calcul diffrential absolu et leurs applications}, Math. Ann. B., {\bf 54} (1901), 125--201.

\bibitem{NOY-jadid}
S. Nasseh, M. Ono, and Y. Yoshino, \emph{Diagonal tensor algebra and na\"{i}ve liftings}, J. Pure Appl. Algebra, {\bf 229} (2025), no.~1, Paper No. 107808.

\bibitem{NOY3}
S. Nasseh, M. Ono, and Y. Yoshino, \emph{Obstruction to na\"{i}ve liftability of DG modules}, J. Commut. Algebra, {\bf 16} (2024), no.~4, 459--475.

\bibitem{NOY2}
S. Nasseh, M. Ono, and Y. Yoshino, \emph{On the semifree resolutions of DG algebras over the enveloping DG algebras}, Comm. Algebra, \textbf{52} (2024), no. 2, 657--667.

\bibitem{NOY1}
S. Nasseh, M. Ono, and Y. Yoshino, \emph{Na\"{i}ve liftings of DG modules}, Math. Z., {\bf 301} (2022), no.~1, 1191--1210.

\bibitem{NOY}
S. Nasseh, M. Ono, and Y. Yoshino, \emph{The theory of $j$-operators with application to (weak) liftings of DG modules}, J. Algebra, {\bf 605} (2022), 199--225.

\bibitem{nasseh:lql}
S. Nasseh and S. Sather-Wagstaff, \emph{Liftings and Quasi-Liftings of DG modules}, J. Algebra, {\bf 373} (2013), 162--182.

\bibitem{nassehyoshino}
S. Nasseh and Y. Yoshino, \emph{Weak liftings of DG modules}, J. Algebra, {\bf 502} (2018), 233--248.

\bibitem{OY}
M. Ono and Y. Yoshino, \emph{A lifting problem for DG modules}, J. Algebra, {\bf 566} (2021), 342--360.

\bibitem{yoshino}
Y. Yoshino, \emph{The theory of {L}-complexes and weak liftings of complexes},
  J. Algebra, \textbf{188} (1997), no.~1, 144--183.
\end{thebibliography}
\providecommand{\bysame}{\leavevmode\hbox to3em{\hrulefill}\thinspace}
\providecommand{\MR}{\relax\ifhmode\unskip\space\fi MR }
\providecommand{\MRhref}[2]{%
  \href{http://www.ams.org/mathscinet-getitem?mr=#1}{#2}
}
\providecommand{\href}[2]{#2}

\end{document}